\documentclass[a4paper]{amsart}
\usepackage{amsmath,amssymb,amsfonts}
\usepackage{mathrsfs,latexsym,amsthm,enumerate}
\usepackage{amscd}
\usepackage[all]{xy}

\newcommand{\D}{\mathcal{D}}

\newcommand{\SkewB}{\mathcal{LSBA}}
\newcommand{\LCBSh}{\mathcal{ESLCBS}}
\newcommand{\BSkewB}{\mathcal{BLSBA}}
\newcommand{\ESBS}{\mathcal{ESBS}}
\newcommand{\LSBIA}{\mathcal{LSBIA}}
\newcommand{\ESBSE}{\mathcal{ESBSE}}
\newcommand{\BLSBIA}{\mathcal{BLSBIA}}
\newcommand{\ESLCBSE}{\mathcal{ESLCBSE}}
\newcommand{\SB}{\bf{SB}}
\newcommand{\Ob}{\mathrm{Ob}}
\newcommand{\Hom}{\mathrm{Hom}}
\newcommand{\ocap}{\overline{\cap}}
\newcommand{\ucup}{\underline{\cup}}

\newcommand{\F}{\mathcal{F}}

\newcommand{\A}{\mathcal{A}}
\newcommand{\B}{\mathcal{B}}
\newcommand{\U}{\mathcal{U}}
\newcommand{\PU}{\mathcal{PU}}

\newtheorem{lemma}{Lemma}
\newtheorem{corollary}{Corollary}
\newtheorem{theorem}{Theorem}

\title{A refinement of Stone duality to skew Boolean algebras}
\author{Ganna Kudryavtseva}
\address{G. K.: University of Ljubljana,
Faculty of Computer and Information Science, \newline
Tr\v{z}a\v{s}ka cesta 25,
SI-1001, Ljubljana,
SLOVENIA.}
\email{ganna.kudryavtseva\symbol{64}fri.uni-lj.si}

\begin{document}
\maketitle
\begin{abstract}
We establish two duality theorems which refine the classical Stone duality between generalized Boolean algebras and locally compact Boolean spaces. In the first theorem we prove that the category of left-handed skew Boolean algebras whose morphisms are proper skew Boolean algebra homomorphisms is equivalent to the category of \'{e}tale spaces over locally compact Boolean spaces whose morphisms are \'{e}tale space cohomomorphisms over continuous proper maps. In the second theorem we prove that the category of left-handed skew Boolean $\cap$-algebras whose morphisms are proper skew Boolean $\cap$-algebra homomorphisms is equivalent to the category of \'{e}tale spaces with compact clopen equalizers over locally compact Boolean spaces whose morphisms are injective \'{e}tale space cohomomorphisms over continuous proper maps.
\end{abstract}
\vspace{0.2cm}


\section{Introduction}

The aim of the present paper is to refine the classical Stone duality \cite{stone37, D} between generalized Boolean algebras and locally compact Boolean spaces to two dualities between skew Boolean algebras and \'{e}tale spaces over locally compact Boolean spaces. In particular, we prove the following statements.

\begin{theorem}\label{th:main} The category $\LCBSh$ of \'{e}tale spaces over locally compact Boolean spaces whose morphisms are \'{e}tale space cohomomorphisms over continuous proper maps is equivalent to the category $\SkewB$ of left-handed skew Boolean algebras whose morphisms are proper skew Boolean algebra homomorphisms.
\end{theorem}

\begin{corollary}\label{cor:main} The category $\ESBS$ of \'{e}tale spaces over Boolean spaces whose morphisms are \'{e}tale space cohomomorphisms is equivalent to the category $\BSkewB$ of left-handed skew Boolean algebras whose maximal lattice images are Boolean algebras and whose morphisms are proper skew Boolean algebra homomorphisms.
\end{corollary}

\begin{theorem}\label{th:main_int}  The category $\ESLCBSE$ of \'{e}tale spaces with compact clopen equalizers over locally compact Boolean spaces whose morphisms are injective \'{e}tale space cohomomorphisms over continuous proper maps is equivalent to the category $\LSBIA$ of left-handed skew Boolean $\cap$-algebras whose morphisms are proper skew Boolean $\cap$-algebra homomorphisms.
\end{theorem}

\begin{corollary}\label{cor:main_int}  The category $\ESBSE$ of \'{e}tale spaces with clopen equalizers over Boolean spaces whose morphisms are injective \'{e}tale space cohomomorphisms is equivalent to the category $\BLSBIA$ of left-handed skew Boolean $\cap$-algebras whose maximal lattice images are Boolean algebras and whose morphisms are proper skew Boolean $\cap$-algebra homomorphisms.
\end{corollary}

Skew Boolean algebras are noncommutative refinements of generalized Boolean algebras in the sense that for a given skew Boolean algebra $S$, there is a canonically defined congruence $\D$ on $S$ such that $S/\D$ is the maximal generalized Boolean algebra homomorphic image of $S$. A dual idea to generalize Boolean algebras is to embed them into bigger objects. Natural examples of such objects are inverse semigroups, as their idempotent sets form lower semilattices. One completes these lower semilattices to Boolean algebras, and the inverse semigroups containing them to noncommutative algebras, generalizing Boolean algebras. This approach was realized in a recent paper \cite{Law}, where there was obtained a duality between Boolean inverse monoids and Boolean groupoids.

A different view of Stone duality for the category of skew Boolean $\cap$-algebras whose morphisms are all skew Boolean $\cap$-algebra homomorphisms has been presented recently at the seminar talk \cite{BC}.

The outline of this paper is as follows. In Section \ref{s:csd} we review the classical Stone duality. Then in Sections \ref{s:presba} and \ref{s:preetale} we collect preliminaries on skew Boolean algebras and \'{e}tale spaces, respectively. Further, in Section \ref{s:frometale} we construct the dual skew Boolean algebra of a given \'{e}tale space over a locally compact Boolean space. The binary operations on this dual skew Boolean algebra are quasi-intersection, quasi-union and quasi-complement of sections, which refine the usual set union, intersection and relative complement operations. In Section \ref{fromsba} we introduce a refinement of the notion of a prime filter of a generalized Boolean algebra and demonstrate that, similarly to as in the classical situation, prime filters of skew Boolean algebras arise as preimages of $1$ under homomorphisms from skew Boolean algebras to the skew Boolean algebra {\bf 3}, which is a non-commutative analogue of the two element Boolean algebra ${\bf{2}}$. The spectrum $S^{\star}$ of a skew Boolean algebra $S$ is defined to be the set of all its prime filters, and the topology of a locally compact Boolean space of $(S/\D)^{\star}$ is refined to the \'{e}tale space topology of $S^{\star}$. In Sections \ref{s:proof1} and \ref{s:proof2} we give the proofs of Theorems \ref{th:main} and \ref{th:main_int}, respectively. We also explain how to transform the \'{e}tale space representation of a skew Boolean algebra into a subdirect product representation. The latter, in the case of skew Boolean $\cap$-algebras, whose maximal lattice images are Boolean algebras, is isomorphic to the Boolean product representation of \cite[4.10]{L3}. Finally, we discuss the connection of our duality for skew Boolean $\cap$-algebras with the generalization of Stone duality for varieties generated by quasi-primal algebras given in \cite{KW}.

\section*{Acknowledgements} The author thanks Andrej Bauer, Karin Cvetko-Vah, Jeff Egger and Alex Simpson for their comments.

\section{Classical Stone duality}\label{s:csd} Here we briefly review the classical Stone duality \cite{stone37, D} between generalized Boolean algebras and locally compact Boolean spaces. The detailed exposition of the duality between Boolean algebras and Boolean spaces can be found in most textbooks on Boolean algebras, e.g., in \cite{BS,giv,johns}. The exposition of the discrete case of the latter duality with an excellent insight to the categorical background can be found in \cite{Awo}.  Let ${\mathcal{GBA}}$ be the category, whose objects are {\em generalized Boolean algebras} (i.e, relatively complemented distributive lattices with a zero) and whose morphisms are proper homomorphisms of generalized Boolean algebras. Recall that a homomorphism $f:B_1\to B_2$ of generalized Boolean algebras is called {\em proper} \cite{D}, provided that for any $c\in B_2$ there exists $b\in B_1$, such that $f(b)\geq c$. By ${\mathcal{BA}}$ we denote the category of Boolean algebras and homomorphisms of Boolean algebras. A topological space $X$ is called {\em locally compact} \cite[p. 335]{giv} if for every point $x$, there is a compact set $K$ whose interior contains $x$. A locally compact Hausdorff space in which the clopen sets form a base is called {\em locally compact Boolean space}. In such a space compact clopen sets form a base for the topology, and every compact open set is clopen. A continuous map $X\to Y$ of topological spaces is called {\em proper} if a preimage of each compact set is compact. Denote by ${\mathcal{LCBS}}$ the category, whose objects are locally compact Boolean spaces and whose morphisms are continuous proper maps between such spaces. If a locally compact Boolean space is compact then it is called {\em Boolean space.} Denote by ${\mathcal{BS}}$ the category whose objects are Boolean spaces and whose morphisms and continuous maps.

By ${\bf 2}=\{0,1\}$ we denote the two-element Boolean algebra. For a generalized Boolean algebra $B$ denote by $B^{\star}$ the set of all prime filters on $B$, that is the set of all preimages of $1$ under nonzero homomorphisms $B\to {\bf 2}$. Let $A\in B$. Set $M(A)=\{F\in B^{\star}: A\in F\}$. The sets $M(A)$ generate a topology on $B^{\star}$, that turns $B^{\star}$ into a locally compact Boolean space. If $X$ is a locally compact Boolean space then the set of all compact clopen sets on $X$ forms, with respect to the binary operations of union, intersection and relative complement of sets and the nullary operation of fixing the empty set, a generalized Boolean algebra, which we denote by $X^{\star}$.

\begin{theorem}[Classical Stone duality, \cite{stone37, D}] The maps $F: {\mathcal{GBA}} \to {\mathcal{LCBS}}$ and $G: {\mathcal{LCBS}}\to {\mathcal{GBA}}$ defined via

$F(B)=B^{\star}$, $B\in \Ob({\mathcal{GBA}})$; $F(f)=f^{-1}$, $f\in \Hom({\mathcal{GBA}})$;

$G(X)=X^{\star}$, $X\in \Ob({\mathcal{LCBS}})$; $G(f)=f^{-1}$, $f\in \Hom({\mathcal{LCBS}})$,\newline
are contravariant  functors, which establish the dual equivalence between the categories ${\mathcal{GBA}}$ and ${\mathcal{LCBS}}$, where the natural isomorphisms $\beta: 1_{{\mathcal{GBA}}}\to GF$ and $\gamma: 1_{{\mathcal{LCBS}}}\to FG$ are given by

$\beta_S(A)=M(A),  S\in \Ob({\mathcal{GBA}}), A\in S$;

$\gamma_X(t)=N_t=\{N\in X^{\star}:t\in N\}, X\in \Ob({\mathcal{LCBS}}), t\in X$.

The restriction of $F$ to the category ${\mathcal BA}$ and the restriction of $G$ to the category ${\mathcal BS}$ establish the dual equivalence between the categories ${\mathcal BA}$ and  ${\mathcal BS}$.
\end{theorem}

\section{Preliminaries on skew Boolean algebras}\label{s:presba}

For the extended introduction to the theory of skew Boolean algebras and skew Boolean $\cap$-algebras we refer the reader to \cite{L2,L3} and \cite{BL}, respectively. To make our exposition self-contained, we collect  below some definitions and basic facts about skew lattices, skew Boolean algebras and skew Boolean $\cap$-algebras.

A {\em skew lattice} $S$ is an algebra $(S;\wedge,\vee)$ of type $(2,2)$, such that the operations $\wedge$ and $\vee$ are associative, idempotent and satisfy the absorption identities $x\wedge(x\vee y)=x=x\vee(x\wedge y)$ and $(y\vee x)\wedge x=x=(y\wedge x)\vee x$. The {\em natural partial order} $\leq$ on a skew lattice $S$ is defined by $x\leq y$ if and only if $x\wedge y=y\wedge x=x$, or equivalently, $x\vee y=y\vee x=y$. A skew lattice $S$ is called {\em symmetric}, if $x\vee y=y\vee x$ if and only if $x\wedge y=y\wedge x$. An element $0$ of a skew lattice $S$ is called a {\em zero}, if  $x\wedge 0=0\wedge x=0$ for all $x\in S$.  A skew lattice $S$ is called {\em Boolean}, if $S$ is symmetric, $S$ has a zero element and each principal subalgebra $\lceil x\rceil=\{u\in S: u\leq x\}=x\wedge S\wedge x$ forms a Boolean lattice.  If $S$ is a Boolean skew lattice and $a,b\in S$, the {\em relative complement} $a\setminus b$ is defined as the complement of $a\wedge b\wedge a$ in the Boolean lattice $\lceil a\rceil$. A {\em skew Boolean algebra} is a Boolean skew lattice, whose signature is enriched by the nullary operation $0$ and the binary relative difference operation, that is, it is an algebra $(S;\wedge,\vee, \setminus,0)$. Skew Boolean algebras satisfy distributivity laws $a\wedge (b\vee c)=(a\wedge b)\vee (a\wedge c)$ and $(b\vee c)\wedge a=(b\wedge a)\vee (c\wedge a)$ \cite[2.5]{L6}.

A skew lattice $S$ is called {\em rectangular} if there exist two sets $L$ and $R$ such that $S=L\times R$, and the operations $\wedge$ and $\vee$ are defined by $(a,b)\wedge (c,d)=(a,d)$ and $(a,b)\vee (c,d)=(c,b)$. Let $\D$ be the equivalence relation on a skew lattice $S$ defined by $x\D y$ if and only if $x\wedge y\wedge x=x$ and $y\wedge x\wedge y=y$. It is known \cite[1.7]{L1} that the relation $\D$ on a skew lattice $S$ is a congruence, the $\D$-classes of $S$ are its maximal rectangular subalgebras, and quotient algebra $S/\D$ forms the maximal lattice image of $S$. If $S$ is a skew Boolean algebra, then $S/\D$ is the maximal generalized Boolean algebra image of $S$ \cite[3.1]{BL}.

A skew lattice $S$ is called {\em left-handed}, if  the equalities $x\wedge y\wedge x=x\wedge y$ and $x\vee y\vee x=y\vee x$ hold for all $x,y\in S$.  In a left-handed skew Boolean algebra the rectangular subalgebras are {\em flat} in the sense that $x\D y$ if and only if $x\wedge y=x$ and $y\wedge x=y$.

A skew Boolean algebra $S$ is called {\em primitive}, if it has only one non-zero $\D$-class, or, equivalently, if $S/\D$ is the Boolean algebra ${\bf 2}$. The skew Boolean algebra ${\bf 3}=\{0,1,2\}$ is a primitive left-handed skew Boolean algebra with the non-zero $\D$-class $\{1,2\}$, the operations on $D$ being determined by lefthandedness: $1\wedge 2=1$, $2\wedge 1=2$, $1\vee 2=2$, $2\vee 1=1$. It is known \cite[Theorem 4.10]{Corn} that up to isomorphism ${\bf 2}$ and ${\bf 3}$ are the only two subdirectly irreducible left-handed skew Boolean algebras.

If $S$ is a skew Boolean algebra and $a\in S$, by $D_a$ we denote the $\D$-class of the element $a$. In the sequel, we will several times use the following easy lemma.

\begin{lemma}\label{lem:unique} Let $S$ be a left-handed skew Boolean algebra and $x,y\in S$ be such that $D_x\geq D_y$.  Then $x\geq x\wedge y$ and $x\vee y\geq y$. Moreover, if $z\in D_y$ is such that $x\geq z$, then $z=x\wedge y$.
\end{lemma}

\begin{proof} That $x\geq x\wedge y$ and $x\vee y\geq y$ is easily checked. Let $z\in D_y$ be such that $x\geq z$. Since $z$ and $x\wedge y$ both belong to the Boolean lattice $\lceil x\rceil$, it follows that they commute under $\wedge$, that is $z\wedge (x\wedge y)=(x\wedge y)\wedge z$. This, together with flatness of $D_y$, implies $z=x\wedge y$.
\end{proof}

Let $S$ be a left-handed skew Boolean algebra and $\alpha: S\to S/\D$ be the canonical projection. It follows from Lemma \ref{lem:unique} that $\lceil c\rceil$ is isomorphic to the Boolean lattice $\lceil \alpha(c)\rceil$ via the map $a\mapsto \alpha(a)$.

\begin{lemma} \label{lem:induced} Let $S_1$, $S_2$ be skew Boolean algebras and $\alpha_1:S_1\to S_1/\D$, $\alpha_2: S_2\to S_2/\D$ the corresponding canonical projections. Let, further, $f:S_1\to S_2$ be a homomorphism. Then $a\D b$ implies $f(a)\D f(b)$. Therefore, $f$ induces a homomorphism ${\overline f}: S_1/\D\to S_2/\D$. Moreover $f$ and ${\overline f}$ are agreed with projections in that $\alpha_2 f={\overline f}\alpha_1$.
\end{lemma}

\begin{proof} $a\D b$ implies $f(a)\wedge f(b)=f(a\wedge b)=f(a)$ and similarly $f(b)\wedge f(a)=f(b)$. It follows that $f(a)\D f(b)$. That $\alpha_2 f={\overline f}\alpha_1$ is verified by an easy calculation.
\end{proof}

We call a homomorphism $f:S_1\to S_2$ of skew Boolean algebras {\em proper,} provided that the induced homomorphism ${\overline f}: S_1/\D\to S_2/\D$ is a proper homomorphism of generalized Boolean algebras.

A skew Boolean algebra $S$ has {\em finite intersections}, if any finite set $\{s_1,\dots s_k\}$ of elements in $S$ has the greatest lower bound with respect to the natural partial order $\leq$, which we call the {\em intersection} of the elements $s_1,\dots, s_k$. A skew Boolean algebra $S$ with finite intersections, considered as an algebra $(S;\wedge,\vee,\setminus,\cap,0)$, where $\cap$ is the binary operation on $S$ sending $(a,b)$ to $a\cap b$, is called a {\em skew Boolean $\cap$-algebra} \cite{BL}.

All skew Boolean algebras, considered in the sequel, are left-handed.

\section{Preliminaries on \'{e}tale spaces}\label{s:preetale}

Here we survey basic standard notions, required for our purposes. They may be found in any textbook on sheaf theory, e.g., \cite{Br,MM}.

Let $X$ be a topological space. A {\em presheaf of sets} over $X$ is a contravariant functor from the category of open sets on $X$ whose morphisms are inclusions, to the category of sets. A {\em sheaf of sets} is a presheaf of sets $\F$ such that for any open covering $U_i$ of an open set $U$ of $X$ and any compatible (that is, agreeing on overlaps) elements $s_i\in \F(U_i)$, these elements can uniquely be glued together to a section of~$U$.

Let $Y$ and $X$ be topological spaces and $f : Y \to X$ be a continuous map. We define a sheaf $\F$ over $X$ by setting for each open set  $\F(U)$ to be the set of all functions $s : U \to Y$ such that $fs = \mathrm{id}_U$. Restriction is given by restriction of functions. This sheaf is called the {\em sheaf of sections} of $Y$ over $X$ with respect to $f$.

Let $\F$ be a sheaf of sets over $X$. The {\em stalk} of $\F$ at a point $x\in X$, denoted $\F_x$, is the set whose elements are equivalence classes of pairs $(U,s)$, where $U$ is an open neighborhood of $x$ and $s\in \F(U)$, and the equivalence relation is given by $(U,s)\sim (V,t)$ if there is an open set $W\subseteq U\cap V$ with $x\in W$ such that the restrictions of $s$ and $t$ to $W$ coincide. The equivalence class of $(U,s)$ is denoted by $s_x$ and is called the {\em germ} of $s$ at $x$.

Let $X$ be a topological space. An {\em \'{e}tale space} $E$ over $X$ is a triple $(E,\pi,X)$, where $E$ is a topological space and $\pi: E \to X$ is a surjective local homeomorphism, called the {\em covering map}. Let $\F$ is a sheaf of sets over $X$. The {\em \'{e}tale space} $E$ of $\F$  is an \'{e}tale space $E=(E,\pi,X)$ such that $\F$ is the sheaf of sections of $E$ over $X$. It is known that all sheaves of sets can be represented as sheaves of sections of their \'{e}tale spaces. Given a sheaf of sets $\F$ over $X$, the \'{e}tale space $E$ of $\F$ is constructed as follows. The points of $E$ are the germs $s_x$, $s\in\F(U)$  and $x \in U$. Then, consequently, $E$ is a disjoint union of the stalks of $\F$ over $X$. The map $\pi:E\to X$ is defined to be the map that sends $\F_x$ to $x$. Any $s\in\F(U)$ induces a map $\tilde{s}:U\to E$, $x\mapsto\tilde{s}(x)$ such that $\pi(\tilde{s}(x))=x$ for all $x\in U$. The topology of $E$ is defined to be the coarsest topology for which all the maps $\tilde{s}$ are continuous. Then $\F$ is isomorphic to the sheaf of section of $E$ over $X$ with respect to $\pi$. The described construction determines an equivalence of categories between the category of sheaves of sets on $X$ and the category of  \'{e}tale spaces over $X$. If $(E,\pi,X)$ is an \'{e}tale space, and $\F$ is the corresponding sheaf of sections, then for any $s\in \F(U)$ the germ $s_x$ is determined by the value $s(x)$. That is why $\F$ is called the {\em sheaf of germs of sections} of its  \'{e}tale space.

Let us fix some notation, used throughout in the sequel. Let $(E,\pi,X)$ be an \'{e}tale space. The points of $E$ will be called {\em germs}. If $U$ is an open set in $X$ then $E(U)$ is the set of all {\em sections} $V$ in $E$ such that $\pi(V)=U$. The {\em stalks} of $E$ are the equivalence classes induced by $\pi$. If $x\in X$ the stalk $S$ in $E$ such that $\pi(s)=x$ for all $s\in S$ will be denoted by $E_x$. If $A\in E(U)$ is a section, then by $A(x)$, $x\in U$, we denote the germs, contained in $A$.

Let $(\A,g,X)$ and $(\B,h,Y)$ be \'{e}tale spaces over $X$ and $Y$, respectively. Recall the definition of an \'{e}tale space cohomomorphism (\cite[p. 14]{Br}). Suppose $f:X\to Y$ is a  continuous map. An $f$-{\em cohomomorphism} $k: \B \rightsquigarrow \A$ is a collection of maps  $k_x:\B_{f(x)}\to \A_x$ for each $x\in X$ such that for every section $s\in \B(U)$ the function $x\mapsto k_x(s(f(x)))$ is a section of $\A$ over $f^{-1}(U)$. An $f$-cohomomorphism $k: \B \rightsquigarrow \A$ is not a function in general, since it is multiply valued unless $f$ is one-to-one, and it is not defined everywhere, unless $f$ is onto.

Let $(E,\pi,X)$ be an \'{e}tale space over a locally compact Boolean space $X$. We call it an {\em \'{e}tale space with compact clopen equalizers}, provided that for every $U,V$ compact clopen in $X$ and any $A\in E(U)$, $B\in E(V)$, the intersection $A\cap B$ is a section, that is there is some compact clopen set $W\subseteq X$ such that $A\cap B\in E(W)$. If $(E,\pi,X)$ is an \'{e}tale space with compact clopen equalizers over a Boolean space $X$ then any clopen set in $X$ is compact, and that is why we call $(E,\pi,X)$ an {\em \'{e}tale space with clopen equalizers}.

All \'{e}tale spaces, considered in the sequel, are \'{e}tale spaces over locally compact Boolean spaces.

\section{The functor $\SB:\LCBSh\to \SkewB$}\label{s:frometale}  Let $X$ be a locally compact Boolean space and $(E,f,X)$ be an  \'{e}tale space.
Our aim now is to define on compact clopen sections of $(E,f,X)$ the quasi-union, quasi-intersection and quasi-complement operations which generalize the usual set union, intersection and relative complement operations and, together with fixing the empty set, define on the union of the sets $E(U)$, where $U$ runs through the compact clopen sets of $X$, the structure of a left-handed skew Boolean algebra.

Fix $U$ and $V$ to be compact clopen sets of $X$ and let $A\in E(U)$, $B\in E(V)$. We define the {\em quasi-union} $A \ucup B$ of $A$ and $B$ to be the section in $E(U\cup V)$ given by
 \begin{equation}\label{def-q-union}
 (A \ucup B)(x)=\left\lbrace\begin{array}{l}B(x), \text{ if } x\in V,\\
 A(x), \text{ if } x\in U\setminus V, \end{array}\right.
 \end{equation}
 and the {\em quasi-intersection} $A \ocap B$ of $A$ and $B$ to be the section in $\F(U\cap V)$ given by
 \begin{equation}\label{def-q-inter}
 (A \ocap B)(x)=A(x) \text{ for all } x\in U\cap V.
 \end{equation}

It is clear that $A \ocap B$ is well defined, since $A \ocap B$ is the restriction of $A$ to $U\cap V$.
To show that $A \ucup B$ is well defined, observe that if $U$ and $V$ are compact clopen sets of $X$, then so is $U\setminus V$ (this easily follows, e.g., from classical Stone duality). So that $A$ restricted to $U\setminus V$ is a section. Then $A \ucup B$ is a section obtained by gluing the restriction of $B$ to $V$ and the restriction of $A$ to $U\setminus V$. Denote by $\varnothing$ the section of the empty set of $X$.

\begin{lemma} \label{sba} $(E,\ucup,\ocap, \varnothing)$ is a left-handed Boolean skew lattice.
\end{lemma}

\begin{proof}
The proof amounts to a routine verification that $(E,\ucup,\ocap, \varnothing)$ satisfies the definition of a left-handed Boolean skew lattice, which we leave to the reader as an exercise.
\end{proof}

 Let $A\in E(U)$, $B\in E(V)$. Then the {\em quasi-complement} $A\setminus B$  is defined as the section in $E(U\setminus V)$ given by
  \begin{equation*}\label{def-q-rel_compl}
 (A \setminus B)(x)=A(x) \text{ for all } x\in U\setminus V.
 \end{equation*}

 We call the left-handed skew Boolean algebra $(E,\ucup,\ocap, \setminus,\varnothing)$ the {\em dual skew Boolean algebra} to the \'{e}tale space $E=(E,f,X)$ and denote it by $E^{\star}=(E,f,X)^{\star}$.

\begin{lemma}\label{lem:parord} The natural partial order on $E^{\star}$ is given by $A\geq B$ if and only if $A\supseteq B$.
\end{lemma}
\begin{proof} The statement follows from the definitions of $\ocap$ and $\leq$.
\end{proof}

\begin{lemma}\label{lem:int} Let $A,B$ be such sections of $E$ that the intersection $A\cap B$  is again a section of $E$. Then $A\cap B$ is the intersection of $A$ and $B$ in $E^{\star}$. Consequently, if $E$ is an \'{e}tale space with compact clopen equalizers then the skew Boolean algebra $E^{\star}$ has finite intersections.
\end{lemma}

\begin{proof} The statement follows from Lemma \ref{lem:parord}.
\end{proof}

In the following lemma we establish the connection between the skew Boolean algebra $(E,f,X)^{\star}$ and the generalized Boolean algebra $X^{\star}$.

\begin{lemma}\label{lem:connection}
Let $X$ be a locally compact Boolean space and $(E,f,X)$ be an \'{e}tale space. Then the $\D$-classes of $(E,f,X)^{\star}$ are the stalks $E_x$, $x\in X$. The maximal generalized Boolean algebra image of $(E,f,X)^{\star}$ is isomorphic to $X^{\star}$, and the canonical projection $\delta: (E,f,X)^{\star}\to X^{\star}$ is given by $V\mapsto U$, whenever $V\in E(U)$, $U\in X^{\star}$.
\end{lemma}

\begin{proof} We use the characterization of the relations $\D$: $A\D B$ if and only if $A\ocap B= A$ and $B\ocap A=B$. The definition of $\ocap$ implies that for $A\in E(U)$ and $B\in E(V)$ we have $A\ocap B =A$ only if $U=V$. In addition, in the latter case both of the equalities $A\ocap B= A$ and $B\ocap A=B$ hold. Now the second statement also follows.
\end{proof}

Let $(E,e,X)$ and $(G,g,Y)$ be \'{e}tale spaces over locally compact Boolean spaces $X$ and $Y$, respectively, $f:X\to Y$ be a continuous proper map and $k:G\rightsquigarrow E$ be an $f$-cohomomorphism. Let $A\in G(U)$. By the definition, the section $k(A)$ is the image of the map sending $x\in X$ to $k_x(A(f(x)))$.

\begin{lemma} $k$ is a proper skew Boolean algebra homomorphism from $G^{\star}$ to $E^{\star}$.
\end{lemma}

\begin{proof} By the classical Stone duality, $f^{-1}: Y^{\star}\to X^{\star}$ is a proper homomorphism of generalized Boolean algebras. Therefore, in view of Lemma \ref{lem:connection}, we are left to show that $k$ preserves $\ocap$, $\ucup$ and $\varnothing$ (then $k$ automatically preserves also $\setminus$). Let $A\in G(U)$, $B\in G(V)$.
First we show that $k(A\ocap B)=k(A)\ocap k(B)$. According to the definition of $\ocap$, $k(A)\ocap k(B)$ is the restriction of $k(A)$ to $f^{-1}(U)\cap f^{-1}(V)$. Further, since $A\ocap B$ equals $A$ restricted to $U\cap V$, then $k(A\ocap B)$ equals $k(A)$ restricted to $f^{-1}(U\cap V)$. It follows that $k(A\ocap B)=k(A)\ocap k(B)$.

Show that $k(A\ucup B)=k(A)\ucup k(B)$. By the definition of $\ucup$, the section $k(A)\ucup k(B)$ is obtained by gluing the restriction of $k(A)$ to $f^{-1}(U)\setminus f^{-1}(V)$ and $k(B)$. Since $A\ucup B$ is obtained by gluing the restriction of $A$ to $U\setminus V$ and $B$, it follows that $k(A\ucup B)$ is obtained by gluing  the restriction of $k(A)$ to $f^{-1}(U\setminus V)$ and $k(B)$. Therefore, $k(A\ucup B)=k(A)\ucup k(B)$.

Finally, since $\varnothing$ is the section of the empty set, it follows that $k(\varnothing)=\varnothing$.
\end{proof}

We can therefore correctly define the functor $\SB$ from the category $\LCBSh$ to the category $\SkewB$ by setting ${\SB}(E,f,X)=(E,f,X)^{\star}$, $(E,f,X)\in\Ob(\LCBSh)$, and ${\SB}(k)=k$, $k\in \Hom(\LCBSh)$.

\section{The functor ${\bf ES}:  \SkewB\to\LCBSh$}\label{fromsba}

Let $S$ be a skew Boolean algebra and $\alpha:S\to S/\D$ be the canonical projection of $S$ onto its maximal generalized Boolean algebra image $S/\D$.

Call a nonempty subset $U$ of $S$ a {\em filter} provided that:
 \begin{enumerate}
\item\label{d1} for all $a,b\in S$: $a\in U$ and $b\geq a$ implies $b\in U$;
\item\label{d2} for all $a,b\in S$: $a\in U$ and $b\in U$ imply $a\wedge b\in U$.
\end{enumerate}

Call a subset $U$ of $S$ a {\em preprime filter} if $U$ is a filter of $S$ and there is a prime filter $F$ of $S/\D$ such that $\alpha(U) =F$. Denote by ${\mathcal PU}_F$ the set of all preprime filters contained in $\alpha^{-1}(F)$. Observe that ${\mathcal PU}_F\neq \varnothing$ since $\alpha^{-1}(F)\in {\mathcal PU}_F$.

\begin{lemma} \label{lem:sym} Let $U\in {\mathcal PU}_F$ and $U \neq \alpha^{-1}(F)$. Then $\alpha^{-1}(F)\setminus U\in {\mathcal PU}_F$.
\end{lemma}

\begin{proof}  $U \neq \alpha^{-1}(F)$ implies $\alpha^{-1}(F)\setminus U \neq \varnothing$. Show first that $\alpha^{-1}(F)\setminus U$ is a filter. Let $a\in \alpha^{-1}(F)\setminus U$ and $b\geq a$. Then $D_b\geq D_a$ and thus $b\in \alpha^{-1}(F)$. Assume that $b\in U$. Let $c\in D_a\cap U$. Then $b\wedge c\in U$ and $b\geq b\wedge c$. By Lemma \ref{lem:unique} we obtain $a=b\wedge c$, implying $a\in U$, that contradicts the choice of $a$. Therefore, $b\in \alpha^{-1}(F)\setminus U$. Let $a,b\in  \alpha^{-1}(F)\setminus U$. If $a\wedge b\in U$ then we would have $a\in U$ since $a\geq a\wedge b$. Therefore, $a\wedge b\in  \alpha^{-1}(F)\setminus U$, and $\alpha^{-1}(F)\setminus U$ is a filter.

Show that any $\D$-class in $F$ has a nonempty intersection with $\alpha^{-1}(F)\setminus U$. Let $D$ be a $\D$-class of $F$ and $a\in \alpha^{-1}(F)\setminus U$. Take any $b\in D$. Then $a\wedge b\in D_a\wedge D$, and $a\wedge b\in \alpha^{-1}(F)\setminus U$ since $a\geq a\wedge b$ and $a\in \alpha^{-1}(F)\setminus U$. Since $D\geq D_a\wedge D$, it follows by Lemma \ref{lem:unique} that there is $c\in D$ such that $c\geq a\wedge b$. But then, since $\alpha^{-1}(F)\setminus U$ is a filter,  $c\in D\cap (\alpha^{-1}(F)\setminus U)$, as required.
\end{proof}

\begin{lemma} \label{cor:hom} Let $f:S\to\bf{3}$ be a nonzero homomorphism. Then there is a prime ideal $M$ of $S/\D$, such that $f^{-1}(0)=\alpha^{-1}(M)$, and a prime filter $F$ of $S/\D$, such that $f^{-1}(\{1,2\})=\alpha^{-1}(F)$.
\end{lemma}

\begin{proof} Let ${\overline f}:S/\D\to {\bf 2}$ be the induced homomorphism of generalized Boolean algebras. Clearly, ${\overline f}$ is nonzero. Let $M={\overline f}^{-1}(0)$. Then $M$ is a prime ideal of $S/\D$ and $\alpha^{-1}(M)=f^{-1}(0)$. The statement follows.
\end{proof}

The following lemma provides a characterization of preprime filters in terms of homomorphisms to the skew Boolean algebra ${\bf 3}$.

\begin{lemma}\label{lem:hom1}
A subset $U\subseteq S$ is a preprime filter if and only if there exists a homomorphism $f:S\to\bf{3}$, such that $U=f^{-1}(1)$.
\end{lemma}

\begin{proof}
 Let $f:S\to {\bf 3}$ be a homomorphism such that $U=f^{-1}(1)$. Let $F$ be a prime filter of $S/\D$ such that $f^{-1}(\{1,2\})=\alpha^{-1}(F)$. It is easy to verify that $U$ is a filter. Show that $\alpha(U) =F$. If this is not so, then there is a $\D$-class $D$ of $F$ such that $D\subseteq f^{-1}(2)$. Let $a\in U$ and $b\in D$. Then $f(a\wedge b)=f(a)\wedge f(b)=1\wedge 2=1$. Observe that $b\geq b\wedge a$ implies $D\geq \D_{b\wedge a}$. Since, in addition, $(b\wedge a)\D (a\wedge b)$  it follows by Lemma \ref{lem:unique} that there exists $t\in D$ such that $t\geq a\wedge b$.
But then $1=f(a\wedge b)=f(t\wedge (a\wedge b))=f(t)\wedge f(a\wedge b)=f(t)\wedge 1$, which implies $f(t)=1$, contradicting $t\in D$. Therefore, $U\in {\mathcal{PU}}_F$.

Suppose now that $U\in {\mathcal{PU}}_F$.  Define a map $f: S\to {\bf 3}$ by setting $f^{-1}(0)=S\setminus \alpha^{-1}(F)$, $f^{-1}(1)=U$ and $f^{-1}(2)=\alpha^{-1}(F)\setminus U$. If $U=\alpha^{-1}(F)$, this map is clearly a homomorphism. Suppose $U\neq \alpha^{-1}(F)$. By Lemma \ref{lem:sym} $\alpha^{-1}(F)\setminus U\in {\mathcal{PU}}_F$. Show that $f(a\wedge b)=f(a)\wedge f(b)$ for any $a,b\in S$. In view of Lemma \ref{lem:sym} it is enough to consider the following cases.

{\em Case 1.} Suppose $a\in S\setminus \alpha^{-1}(F)$ and $b\in \alpha^{-1}(F)$. Then $D_a\in S/\D\setminus F$ and $D_b\in F$. It follows that $D_{a\wedge b}=D_a\wedge D_b\in S/\D\setminus F$, which implies that $a\wedge b\in S\setminus \alpha^{-1}(F)$ and thus $f(a\wedge b)=0=0\wedge f(b)=f(a)\wedge f(b)$.

{\em Case 2. }Suppose $a\in \alpha^{-1}(F)\setminus U$ and $b\in U$. Since $a\geq a\wedge b$ and $\alpha^{-1}(F)\setminus U$ is a filter it follows that $a\wedge b \in \alpha^{-1}(F)\setminus U$. Similarly, $b\wedge a \in U$. Therefore, $f(a\wedge b)=f(a)\wedge f(b)$ and $f(b\wedge a)=f(b)\wedge f(a)$.

{\em Case 3.} Suppose $a,b\in U$. Since $a\geq a\wedge b$ it follows that $a\wedge b\in U$ and similarly $b\wedge a\in U$. Therefore, $f(a\wedge b)=f(a)\wedge f(b)$ and $f(b\wedge a)=f(b)\wedge f(a)$.

Applying similar arguments, one shows that $f$ respects also $\vee$. Since, in addition, $f(0)=0$, it follows that $f$ is a homomorphism of skew Boolean algebras.
\end{proof}

 Let us call minimal elements of ${\mathcal{PU}}_F$, where $F$ runs through the set of prime filters of the Boolean algebra $S/\D$, {\em prime filters} of the skew Boolean algebra $S$. Applying Lemma \ref{lem:hom1} we see that prime filters are exactly minimal nonempty preimages of $1$ under the morphisms $S\to {\bf 3}$. We denote the set of all prime filters $U$, such that $\alpha(U) =F$, by ${\mathcal U}_F$. To distinguish between prime filters of generalized Boolean algebras and prime filters of skew Boolean algebras, we will refer to the former as to BA-prime filters and to the latter as to SBA-prime filters. The following lemma shows that SBA-prime filters exist, and that every non-zero element of $S$ belongs to some of them.

\begin{lemma}\label{lem:ultra} Let $a\in S$, $a\neq 0$, and let $F$ be a BA-prime filter of $S/\D$, such that $D_a\in F$ (it exists by the classical Stone duality). Then the set
$$
X_{a,F}=\{s\in S: s\geq t \text{ for some } t\in \alpha^{-1}(F) \text{ such that } t\leq a\}
$$
is a SBA-prime filter of $S$, contained in ${\mathcal U}_F$ and containing $a$.
\end{lemma}
\begin{proof}
Show that $X_{a,F}\subseteq \alpha^{-1}(F)$. Let $s\in X_{a,F}$ and $t\in \alpha^{-1}(F)$ be such that $t\leq a$ and $s\geq t$. Then $D_s\geq D_t$, and $D_t\in F$. It follows that $D_s\in F$.

 Now we verify that $X_{a,F}$  satisfies Conditions \ref{d1} and \ref{d2} of the definition of a filter. If $c\in X_{a,F}$ and $d\geq c$, then obviously $d\in X_{a,F}$. Suppose $c,d\in X_{a,F}$. Then there are some $t_1,t_2\in  \alpha^{-1}(F)$ such that $t_1\leq a$, $t_2\leq a$ and $c\geq t_1$, $d\geq t_2$. Since $\lceil a\rceil$ is a Boolean algebra and $t_1,t_2\in \lceil a\rceil$, it follows that $t_1\wedge t_2=t_2\wedge t_1$.
Then we have
\begin{multline*}
(c\wedge d)\wedge (t_1\wedge t_2)=(c\wedge d)\wedge (t_2\wedge t_1)=c\wedge (d\wedge t_2)\wedge t_1= \\
c\wedge t_2\wedge t_1=c\wedge t_1\wedge t_2=t_1\wedge t_2
\end{multline*}
and similarly $(t_1\wedge t_2)\wedge(c\wedge d)=t_1\wedge t_2$. Hence, $c\wedge d\geq t_1\wedge t_2$. Since $t_1\leq a$ and $t_2\leq a$, it follows that $t_1\wedge t_2\leq a$. Since $t_1\in  \alpha^{-1}(F)$ and $t_2\in  \alpha^{-1}(F)$, we have $t_1\wedge t_2\in  \alpha^{-1}(F)$. This shows that $c\wedge d\in X_{a,F}$.

Now we show that $X_{a,F}$ has nonempty intersections with all $\D$-classes of $F$. Let $D$ be a $\D$-class of $F$. Consider the $\D$-class $D'=D_a\wedge D$. We have $D'\in F$ and $D_a\geq D'$. Let $c\in D'$. Then $a\wedge c\in \alpha^{-1}(F)$, $a\wedge c\in D_a\wedge D_c=D_c=D'$and $a\geq a\wedge c$. Let $d\in D$. Since $D\geq D'$ it follows from Lemma \ref{lem:unique} that $d\vee(a\wedge c)\geq a\wedge c$. This implies that $d\vee (a\wedge c)\in X_{a,F}\cap D$, and therefore $X_{a,F}\cap D\neq\varnothing$. Hence $X_{a,F}\in {\mathcal{PU}}_F$.

Finally, we show that $X_{a,F}\in {\mathcal U}_F$. Assume $U\in {\mathcal PU}_F$ be such that $U\subseteq X_{a,F}$. Let $b\in U\cap D_a$. Since $b\in X_{a,F}$, there exists $t\leq a$, $t\in \alpha^{-1}(F)$, such that $b\geq t$. Let $s\in U\cap D_t$. Then $b\wedge s\in U\cap D_t$ and $b\geq b\wedge s$. By Lemma \ref{lem:unique} it follows that $b\wedge s=t$. This implies $t\in U$, and hence $a\in U$.

Let $t\leq a$ and $t\in \alpha^{-1}(F)$. Show that $t\in U$. Let $y$ be such that $y\in U\cap D_t$. Then $a\wedge y\in U$ and $a\geq a\wedge y$. By Lemma \ref{lem:unique}, $a\wedge y=t$, which implies that $t\in U$, as required. Now it follows that $X_{a,F}\subseteq U$. Then we obtain  $X_{a,F}= U$, implying $X_{a,F}\in {\mathcal U}_F$.
\end{proof}

In the following lemma we establish the property of SBA-prime filters, which is crucial for the construction of the dual \'{e}tale space, which will follow.

\begin{lemma}\label{prop_ultra} Let $F$ be a BA-prime filter of $S/\D$ and $U_1,U_2\in  {\mathcal U_F}$. Then either $U_1=U_2$ or $U_1\cap U_2=\varnothing$.
\end{lemma}

\begin{proof} Assume $U_1,U_2\in  {\mathcal U_F}$ are such that $U_1\cap U_2\neq\varnothing$. First we show that $U_1\cap U_2\in {\mathcal{PU}}_F$. It is clear that $U_1\cap U_2$ is a filter. So it is enough to demonstrate that $\alpha(U_1\cap U_2) =F$. Since $\alpha(U_1)=F$ and  $\alpha(U_2)=F$, it follows that $\alpha(U_1\cap U_2)\subseteq F$. Let $D$ be a $\D$-class of $F$. We aim to show that $D$ has a non-empty intersection with $U_1\cap U_2$. Fix $a\in U_1\cap D$ and $b\in U_2\cap D$ (such $a$ and $b$ exist since $\alpha(U_1)=F$ and $\alpha(U_2)=F$). Fix also $c\in U_1\cap U_2$ ($c$ exists since $U_1\cap U_2\neq\varnothing$ by the assumption). Consider the elements $c\wedge a$ and $c\wedge b$. Since $a\D b$ and since $\D$ is a congruence, it follows that $(c\wedge a)\D (c\wedge b)$. Since $c\geq c\wedge a$, $c\geq c\wedge b$ and $(c\wedge a)\D (c\wedge b)$, it follows from Lemma \ref{lem:unique} that $c\wedge a=c\wedge b$.
Since $\D$ is a congruence, then $(a\wedge c)\D (c\wedge a)$
  and $(b\wedge c)\D (c\wedge b)$. Hence, the three elements $c\wedge a=c\wedge b$, $a\wedge c$ and $b\wedge c$ lie in the same $\D$-class. Let $x=a\setminus (a\wedge c)$. Since $a\geq a\wedge c$, it follows that $x$ is the complement of $a\wedge c$ in the Boolean algebra $\lceil a\rceil$, that is, $(a\wedge c)\wedge x=x\wedge (a\wedge c)=0$ and $(a\wedge c)\vee x= x\vee (a\wedge c)=a$. Set $y=(c\wedge a)\vee x$. Since $(c\wedge a)\D (a\wedge c)$, then $y\D ((a\wedge c)\vee x)$, that is $y\D a$. Applying distributivity, idempotency, the property of $0$ and absorption, we see that
$$
y\wedge (c\wedge a)=((c\wedge a)\vee x)\wedge (c\wedge a)=((c\wedge a)\wedge(c\wedge a))\vee (x\wedge(c\wedge a))=(c\wedge a)\vee 0= c\wedge a;
$$
$$
(c\wedge a)\wedge y=(c\wedge a)\wedge ((c\wedge a)\vee x)=c\wedge a.
$$
So, $y\geq c\wedge a=c\wedge b\in U_1\cap U_2$. It follows that $y\in U_1\cap U_2$. Therefore, $y\in (U_1\cap U_2)\cap D$, and thus $(U_1\cap U_2)\cap D\neq \varnothing$, as required.
\end{proof}

\begin{corollary}\label{cor:union}
Let $F$ be a BA-prime filter of $S/\D$. Than $\alpha^{-1}(F)$ is a disjoint union $\alpha^{-1}(F)=\dot\cup_{U\in {\mathcal U}_F}U$. Moreover, any preprime filter $U\subseteq \PU(F)$ is a disjoint union of prime filters from $\U_F$, contained in $U$.
\end{corollary}
\begin{proof} The statements follow from Lemma~\ref{prop_ultra} and Lemma~\ref{lem:ultra}.
\end{proof}

\begin{corollary}\label{cor:new} Let $F$ be a BA-prime filter of $S/\D$ and $U\in\U_F$. Then $U=X_{a,F}$ for any $a\in U$.
\end{corollary}
\begin{proof} Let $a\in \alpha^{-1}(F)$. Since $a\in X_{a,F}$, it follows that $\alpha^{-1}(F)$ is covered by the sets $X_{a,F}$. Therefore, $U=X_{a,F}$ for some $a\in U$. Let $b\in X_{a,F}$. Then $X_{a,F}\cap X_{b,F}\neq\varnothing$, which in view of Lemma~\ref{prop_ultra} and Lemma~\ref{lem:ultra} implies that $X_{a,F}=X_{b,F}$.
\end{proof}

\begin{lemma}\label{lem:preim} Let $S_1,S_2$ be skew Boolean algebras, $k:S_1\to S_2$ be a proper skew Boolean algebra homomorphism and $\overline{k}:S_1/\D\to S_2/\D$ be the induced proper homomorphism of generalized Boolean algebras. Let, further, $U$ be a preprime filter of $S_2$ and $F$ be a BA-prime filter of $S_2/\D$ such that $U\in \PU_F$. Then either $k^{-1}(U)=\varnothing$ or $k^{-1}(U)\in \PU_{\overline{k}^{-1}(F)}$.
\end{lemma}
\begin{proof} Suppose $k^{-1}(U)\neq\varnothing$. Let $f:S_2\to {\bf 3}$ be the homomorphism such that $U=f^{-1}(1)$ (such an $f$ exists by Lemma \ref{lem:hom1}). Then $fk$ is a homomorphism from $S_1$ to {\bf 3} and $k^{-1}(U)=(fk)^{-1}(1)$.  Then by Lemma~\ref{lem:hom1}, $k^{-1}(U)$ is a preprime filter of $S_1$. Since $k$ and $\overline{k}$ are agreed with projections, it follows that $k^{-1}(U)\in \PU_{\overline{k}^{-1}(F)}$.
\end{proof}

Let $S^{\star}$ be the set of all SBA-prime filters of $S$. We call it the {\em spectrum} of $S$. Let $f:S^{\star}\to (S/\D)^{\star}$ be the map, given by $U\mapsto F$, whenever $U\in\U_F$, where $F\in (S/\D)^{\star}$. Obviously, $f$ is surjective. In the classical Stone duality it is proved that $(S/\D)^{\star}$ is a locally compact Boolean space, whose base constitute the sets
\begin{equation}\label{m_prime}
M'(A)=\{F:  F \text{ is a BA-prime filter of } S/\D \text{ and } A\in F\},
\end{equation}
where $A$ runs through $S/\D$. Moreover, all the sets $M'(A)$ are compact clopen, and every compact clopen set coincides with some $M'(A)$.

Let $a\in S$. We define the set
\begin{equation}\label{def_ma}
M(a)=\{F: F \text{ is a SBA-prime filter of } S \text{ and } a\in F\}.
\end{equation}

Let $a\in S$, $a\neq 0$, and $F$ be a BA-prime filter of $S/\D$ such that $\alpha(a)\in F$. Then by Lemma \ref{lem:ultra}, $a\in X_{a,F}$. Moreover, $X_{a,F}$ is a unique  SBA-prime filter, contained in ${\mathcal U}_F$ and containing $a$, as follows from Lemma~\ref{prop_ultra}. Thus we can define a map $s_a: M'(\alpha(a))\to  M(a)$ by setting $F\mapsto X_{a,F}$, $F\in M'(\alpha(a))$. It is clear that $f(s_a(F))=F$ for any $F\in M'(\alpha(a))$.

Define a topology on $S^{\star}$ so that it is the coarsest topology under which all maps $s_a$ are open maps, that is the topology whose subbase constitute the sets $M(a)$, $a\in S$.

\begin{lemma} \label{lem:pi} $(S^{\star},f, (S/\D)^{\star})$ is an \'{e}tale space.
\end{lemma}

\begin{proof}  Let $U\in S^{\star}$ and $a\in U$. Then $M(a)$ is an open neighborhood of $U$, and $f(M(a))=M'(\alpha(a))$. Show that the restriction of $f$ to $M(a)$ is a homeomorphism. Let $M(b)\subseteq M(a)$. By the definitions, this implies that $M'(\alpha(b))\subseteq M'(\alpha(a))$. Then $f(M(b))=M'(\alpha(b))$ is an open set in $(S/\D)^{\star}$, proving that the restriction of $f$ to $M(a)$ is an open map. Let $M'(B)\subseteq M'(\alpha(a))$. Then $B\leq \alpha(a)$ in $S/\D$. Let $b\in B$ be such that $b\leq a$. If $F\in M'(B)$, then $a\in X_{b,F}$, which implies $X_{a,F}=X_{b,F}$ by Corollary \ref{cor:new}. Therefore, $M(b)\subseteq M(a)$. It follows that $f^{-1}(M'(B))\cap M(a)=M(b)$, proving that the restriction of $f$ to $M(a)$ is a continuous map. Finally, the restriction of $f$ to $M(a)$ is a bijection, since $M(a)=\{X_{a,F}: F\in M(\alpha(a))\}$ and $f$ maps $X_{a,F}$ to $F$. Therefore, $f$ is a local homeomorphism.
\end{proof}

Call $S^{\star}=(S^{\star},f, (S/\D)^{\star})$ the {\em dual \'{e}tale space} to the skew Boolean algebra $S$.

\begin{lemma}\label{lem:last} Suppose $S$ is a skew boolean algebra with finite intersections. Then $S^{\star}$ is an \'{e}tale space with compact clopen equalizers.
\end{lemma}

\begin{proof} For $a,b\in S$ we have
$$M(a)\cap M(b)=\{F: F \text{ is a SBA-prime filter of } S \text{ and } a\cap b\in F\}=M(a\cap b),$$
where $a\cap b$ denotes the intersection of $a$ in $b$ in $S$. It follows that
the intersection $M(a)\cap M(b)$ is a section.
\end{proof}

Let $S,T$ be skew Boolean algebras and $k:T\to S$ be a proper skew Boolean algebra homomorphism. Let $\overline{k}:T/\D\to S/\D$ be the induced proper homomorphism of generalized Boolean algebras (it is well defined by Lemma \ref{lem:induced}).
By the classical Stone duality $\overline{k}^{-1}$ is a continuous proper map from $(S/\D)^{\star}$ to $(T/\D)^{\star}$. Denote $g=\overline{k}^{-1}$. Let $S^{\star}=(S^{\star},f_1, (S/\D)^{\star})$  and $T^{\star}=(T^{\star},f_2, (T/\D)^{\star})$ to be the corresponding dual \'{e}tale spaces. Let $F\in (S/\D)^{\star}$ and $V\in S^{\star}_F=\U_F$. By Lemma \ref{lem:preim} $k^{-1}(V)$, if nonempty, is some $U'\in\PU_{g(F)}$. By Corollary \ref{cor:union} $U'$ is a disjoint union of some SBA-prime filters from $\U_{g(F)}$. We set $k_F(U)=V$, provided that $U$ is a prime filter in $\U_{g(F)}=T^{\star}_{g(F)}$, $U\subseteq U'$ and $k^{-1}(V)=U'$. In this way we have defined a map $k_{F}$ from $T^{\star}_{g(F)}$ to  $S^{\star}_F$.

\begin{lemma} All $k_F$ are homomorphisms of \'{e}tale spaces. Moreover, the collection of homomorphisms $k_F$, $F\in (S/\D)^{\star}$, constitutes a $\overline{k}^{-1}$-cohomomorphism $\tilde{k}: T^{\star}\rightsquigarrow S^{\star}$.
\end{lemma}

\begin{proof} By the construction, the maps $k_F$ are agreed with projections, and they are continuous, since stalks have discrete topology. Denote $g=\overline{k}^{-1}$. Let $U\in (T/\D)^{\star}$, $B\in T^{\star}(U)$ and $F\in g^{-1}(U)$. Then
$k_F(B(g(F)))\in S^{\star}_F$. This and the fact that the map $k_F\cdot B\cdot g$ is continuous as a composition of continuous maps prove that the map $F\mapsto k_F(B(g(F)))$, $F\in g^{-1}(U)$, is a section of $S^{\star}$ over $g^{-1}(U)$, as required.
\end{proof}

It follows that setting ${\bf ES}(S)=S^{\star}$, $S\in\Ob(\SkewB)$, and ${\bf ES}(k)=\tilde{k}$,  $k\in\Hom(\SkewB)$, defines a functor ${\bf ES}$ from the category $\SkewB$ to the category $\LCBSh$.

\section{Proof of Theorem \ref{th:main}}\label{s:proof1}

We are going to prove that the functors ${\bf{ES}}:\SkewB\to\LCBSh$ and ${\SB}:\LCBSh\to \SkewB$ establish the required equivalence of categories, where the natural isomorphism $\beta: 1_{\SkewB}\to {\SB}\cdot{\bf{ES}}$ and $\gamma: 1_{\LCBSh}\to {\bf{ES}}\cdot{\SB}$ are given by
\begin{equation}\label{eq:beta} \beta_S(a)=M(a), S\in \Ob(\SkewB), a\in S;
\end{equation}
\begin{equation}\label{eq:gamma} \gamma_{E}(A)=N_A=\{N\in E^{\star}: A\in N\}, E\in \Ob(\LCBSh), A\in E.
\end{equation}

We have already shown that ${\bf{ES}}:\SkewB\to\LCBSh$ and ${\SB}:\LCBSh\to \SkewB$ are well-defined functors. We split the rest of the proof into two lemmas.

\begin{lemma} $\beta: 1_{\SkewB}\to {\bf{ES}}\cdot {\SB}$ is a natural isomorphism.
\end{lemma}
\begin{proof}
Let $S\in \Ob(\SkewB)$ and $a\in S$. Then $M(a)\in S^{\star}(M'(\alpha(a)))$ and thus $M(a)\in S^{\star\star}$. It follows that the map $\beta_S$ is well-defined.

According to \cite[Lemma 7.11, p.165]{Awo} it is enough to show that $\beta$ is a natural transformation and that each $\beta_S$ is an isomorphism. The construction of the functors ${\bf{ES}}$ and ${\SB}$ immediately implies that for any $f\in \Hom_{\SkewB}(S,S')$ we have $\beta_{S'}\cdot f=({\SB}\cdot{\bf{ES}})(f)\cdot \beta_S$. We are left to show that each  $\beta_S$ is a skew Boolean algebra isomorphism.

Fix $S\in \Ob(\SkewB)$. Let us verify that $\beta_S$ is one-to-one. Let $a,b \in S$ and $a\neq b$. Without loss we assume $a\neq 0$. We aim to show that $M(a)\neq M(b)$. Assume first that $D_a\neq D_b$. By the classical Stone duality, there is a BA-prime filter $F$ of $S/\D$ such that $D_a\in F$ and $D_b\not\in F$. It follows that $b\not\in X_{a,F}$, which implies $M(a)\neq M(b)$.

Assume now that $D_a=D_b$. Since $a\neq 0$, then $D_a$ is a nonzero $\D$-class of $S$. Let
$$
C=\{c\in S: c\leq a \text{ and } c\leq b\}.
$$
The set $C$ is not empty, since $0\in C$. Show that $\alpha(C)$ is an ideal of the Boolean algebra $S/\D$. Let $D\in \alpha(C)$ and $D'\leq D$. Take $d\in D\cap C$ and $c\in D'$. We have $d\wedge c\in D'$ and by Lemma \ref{lem:unique} $d\geq d\wedge c$. Since $a\geq d$ and $b\geq d$ it follows that $a\geq d\wedge c$ and $b\geq d\wedge c$, which implies $d\wedge c\in C$, and thus $D'\in \alpha(C)$. Now let $D_1, D_2 \in \alpha(C)$. Show that $D_1\vee D_2\in \alpha(C)$. Let $c\in D_1\cap C$ and $d\in D_2\cap C$. Then $c\vee d\in D_1\vee D_2$. In addition, since $c\leq a$ and $d\leq a$, we obtain
$$
(c\vee d)\wedge a= (c\wedge a)\vee (d\wedge a)=c\vee d;
$$
$$
a\wedge (c\vee d)= (a\wedge c)\vee (a\wedge d)=c\vee d.
$$
Therefore, $c\vee d\leq a$. Similarly one shows that $c\vee d\leq b$. Hence $c\vee d\in C$ and thus $D_1\vee D_2\in \alpha(C)$.

 From the definition of $C$ and from $D_a=D_b$, it follows that $D_a\not\in \alpha(C)$. By Birkhoff prime ideal theorem, there is a prime ideal $M$ of $S/\D$, such that $D_a\not\in M$ and $M\supseteq \alpha(C)$. Denote $F=S/\D\setminus M$. Then $F$ is a BA-prime filter of $S/\D$ such that $D_a\in F$ and $F\cap M=\varnothing$. In addition,  $X_{a,F}\cap C=\varnothing$.
If $b\in X_{a,F}$ then there exists  $t\in \alpha^{-1}(F)$ such that $t\leq a$ and $b\geq t$. But then $t\in C$, which contradicts $X_{a,F}\cap C=\varnothing$. Therefore, $b\not\in X_{a,F}$, which implies $M(a)\neq M(b)$.

Let us verify that $\beta_S$ is onto. Let $U$ be a compact clopen set in $(S/\D)^{\star}$. By the classical Stone duality, there is $A\in S/\D$ such that $U=M'(A)$. By the definition of the topology on $S^{\star}$, $S^{\star}(M'(A))=\{M(a): a\in A\}$, which implies that $\beta_S$ is onto.

Finally, we verify that $\beta_S$ is a skew Boolean algebra homomorphism (it is then automatically a proper homomorphism since it is a bijection). Let $a,b\in S$. Show that $M(a\wedge b)=M(a)\ocap M(b)$. We have
$$M(a\wedge b)\in S^{\star}(M'(\alpha(a\wedge b)))=S^{\star}(M'(\alpha(a)\wedge\alpha(b))).$$
By the classical Stone duality,  $M'(\alpha(a)\wedge\alpha(b))=M'(\alpha(a))\cap M'(\alpha(b))$. Therefore,
$$M(a\wedge b)\in S^{\star}(M'(\alpha(a))\cap M'(\alpha(b))).$$ Let $F\in M'(\alpha(a))\cap M'(\alpha(b))$. Then $(M(a\wedge b))(F)=X_{a\wedge b,F}$.

Since $M(a)\in S^{\star}(M'(\alpha(a)))$ and $M(b)\in S^{\star}(M'(\alpha(b)))$, then $M(a)\ocap M(b)\in S^{\star}(M'(\alpha(a))\cap M'(\alpha(b)))$. Let $F\in M'(\alpha(a))\cap M'(\alpha(b))$. Then $(M(a)\ocap M(b))(F)=(M(a))(F)=X_{a,F}$. Since $a\geq a\wedge b$, it follows that $a\in X_{a\wedge b,F}$, which, together with Corollary \ref{cor:new}, implies $X_{a, F}=X_{a\wedge b,F}$. This completes the proof of the equality $M(a\wedge b)=M(a)\ocap M(b)$. That $M(a\vee b)=M(a)\ucup M(b)$ is proved in the same manner. Finally, it is clear that $M(0)=\varnothing$.
\end{proof}

Let $(E,f,X)$ and $(G,g,Y)$ be \'{e}tale spaces. They are called {\em isomorphic}, provided that there exist homeomorphisms $\varphi:E\to G$ and $\psi:X\to Y$, such that $g \cdot\varphi=\psi\cdot f$.

\begin{lemma} $\gamma: 1_{\LCBSh}\to {\bf{ES}}\cdot{\SB}$ is a natural isomorphism.
\end{lemma}
\begin{proof}
We start from the verification that the maps  $\gamma_{E}$ are well-defined. Let $E=(E,f,X)\in \Ob(\LCBSh)$ and $a\in E$. We aim to show that $N_a=\{N\in E^{\star}: a\in N\}$ is an SBA-prime filter of $E^{\star}$. By the classical Stone duality the set $N'_{f(a)}$ is a BA-prime filter of $X^{\star}$. First we establish that $N_a\in\PU_{N'_{f(a)}}$. A direct verification shows that $N_a$ is a filter.  Suppose $N\in E^{\star}$ is such that $a\in N$. Then $f(a)\in f(N)$, which implies that $f(N)\in N'_{f(a)}=\{N'\in X^{\star}:f(a)\in N'\}$. It follows that $N_a\subseteq f^{-1}(N'_{f(a)})$. We are left to verify that $N_a\cap N'\neq \varnothing$ whenever $N'$ is a compact clopen subset of $X$, such that $f(a)\in N'$. Since $f$ is a local homeomorphism, it follows that there is a compact clopen subset $U$ of $N'$ such that $f(a)\in U$ and an open neighborhood $V\in E(U)$ of $a$ such that the restriction of $f$ to $V$ maps it homeomorphically onto $U$. Applying the fact that $N'\setminus U$ is also compact clopen set (this follows from the classical Stone duality) and the gluing axiom, we glue together $V$ and any $V'\in E(N'\setminus U)$ to obtain a section $V''\in E(U)$, such that $a\in V''$, as required.

Fix some $N\in N_a$. Now we aim to show that $N_a=X_{N, N'_{f(a)}}$.  In view of  $N_a\in\PU(N'_{f(a)})$, it is enough to establish the inclusion $N_a\subseteq X_{N, N'_{f(a)}}$. Let $M\in N_a$. Since $a\in M\cap N$ and since $f$ is a local homeomorphism, it follows that there is a compact clopen neighborhood $W$ of $a$ such that the restrictions of $M$ and $N$ to $W$ coincide. Denote these restrictions by $T$. We have $T\in N_a\subseteq f^{-1}(N'_{f(a)})$, and by Lemma \ref{lem:parord} also that $N\geq T$ and $M\geq T$. It follows that $M\in X_{N, N'_{f(a)}}$. This finishes the proof that $N_a$ is an SBA-prime filter of $E^{\star}$, and that is why $\gamma_{E}$ is well defined.

According to \cite[Lemma 7.11, p.165]{Awo}, to show that $\gamma$ is a natural isomorphism, it is enough to verify that $\gamma$ is a natural transformation and that each $\gamma_{E}$ is an isomorphism. Let $E=(E,f,X)\in \Ob(\LCBSh)$.  The construction of the functors ${\bf{ES}}$ and ${\SB}$ immediately implies that for any $f\in \Hom_{\LCBSh}(E,E')$ we have $\gamma_{E'}\cdot f=({\bf{ES}}\cdot{\SB})(f)\cdot \gamma_{E}$. We are left to show that each $\gamma_{E}$ is an isomorphism of \'{e}tale spaces.

First we show that  $\gamma_{E}$ is one-to-one. Suppose $a,b\in E$ and $a\neq b$. If $f(a)\neq f(b)$ then $N_a\neq N_b$ is a consequence of $N'_{f(a)} \neq N'_{f(b)}$, which follows from the classical Stone duality. Suppose $f(a)=f(b)$ and let $N\in N_a$. Since $a$ and $b$ belong to the same stalk of $E$ and since $N$ is a section, it follows that $b\not\in N$, so that $N\not\in N_b$. Hence $N_a\neq N_b$ in this case as well.

We proceed to show that $\gamma_{E}$ is onto. Let $U$ be an SBA-prime filter of $E^{\star}$ and $F$ be a BA-prime filter of $X^{\star}$ such that $U\in \U_F$. By the classical Stone duality, there is $A\in X$ such that $F=N'_A$.  For any $a\in E_A$ (that is such an $a$ that $f(a)=A$) we have $N_a\in\U_{N'_A}$, as follows from the the above proved equality $N_a=X_{N, N'_{f(a)}}$, $a\in N$. By Corollary \ref{cor:new}, the SBA-prime filters $X_{N,N'_{f(a)}}$ exhaust all the set $\U_{N'_A}$. Thus $N_a\in \U_{N'_A}$.

Finally, we specify an isomorphism between $E$ and $E^{\star\star}$. Lemma \ref{lem:connection} and the definition of the dual \'{e}tale space imply that the covering map $\tilde{f}:E^{\star\star}\to X^{\star\star}$ is given by $N_a\mapsto N'_{f(a)}$, $a\in E$. By the classical Stone duality, $X$ is homeomorphic to $X^{\star\star}$ via the map $\varphi$ which sends $a\in X$ to $N'_a=\{N'\in X^{\star}:a\in N'\}\in X^{\star\star}$. Let $(E,f,X)\in \Ob(\LCBSh)$. Then it is easily seen that $\tilde{f}\cdot\gamma_E=\varphi \cdot f$. We are left to show that $\gamma_{E}$ is open and continuous.

 By the classical Stone duality, an open set in $X^{\star\star}$ looks like $\tilde{A}=\{F\in X^{\star\star}:A\in F\}$, where $A$ is a compact clopen set of $X$. It follows that the sets $\{U\in E^{\star\star}:a\in U\}$, where $a$ runs through $E(A)$ are all sections in $E^{\star\star}(\tilde{A})$.
 Let $A$ be a compact clopen set of $X$ and $a\in E(A)$. We have
 $$ \gamma_{E}(a)=\{U\in E^{\star\star}: \gamma_E(x)=U {\text{ for some }} x\in a\}=\{U\in E^{\star\star}:a\in U\},
 $$
implying that $\gamma_{E}$ is an open map. Since $\gamma_{E}$ is a bijection, we have also $\gamma_{E}^{-1}(\{U\in E^{\star\star}:a\in U\})=a$, implying that $\gamma_{E}$ is also continuous. This completes the proof.
\end{proof}

Corollary \ref{cor:main} follows now from Theorem \ref{th:main} and the classical Stone duality between Boolean algebras and Boolean spaces.

Let $S$ be a left-handed skew Boolean algebra. For every stalk $S^{\star}_x$, $x\in (S/\D)^{\star}$, let $\tilde{S^{\star}_x}$ be the primitive left-handed skew Boolean algebra, whose non-zero $\D$-class equals ${S^{\star}_x}$. Then $S$ can be represented as a subdirect product of $\prod_{x\in (S/\D)^{\star}}\tilde{S^{\star}_x}$ in the way, which refines the representation of a generalized Boolean algebra $S/\D$ into ${\bf 2}^{(S/\D)^{\star}}$: if $V\in S^{\star}(A)$ is a section, it is mapped to $(v_x)_{x\in (S/\D)^{\star}}$, where $v_x=V(x)$, if $x\in A$ and $v_x=0$, otherwise.

\section{Proof of Theorem \ref{th:main_int}}\label{s:proof2}

Let $(E,f,X)$ be an \'{e}tale space with compact clopen equalizers. Then by Lemma~\ref{lem:int} the dual skew Boolean algebra $(E,f,X)^{\star}$ has finite intersections, and we can consider it as a skew Boolean $\cap$-algebra. Conversely, given a skew Boolean $\cap$-algebra $S$ it follows from Lemma \ref{lem:last} that the \'{e}tale $S^{\star}$ space has compact clopen equalizers.

\begin{lemma}\label{31jan1} Let $S,T$ be skew Boolean algebras with finite intersections and $k:S\to T$ a proper skew Boolean algebra homomorphism that preserves finite intersections. Then $\tilde{k}:S^{\star}\rightsquigarrow T^{\star}$ is an injective \'{e}tale space cohomomorphism over a continuous proper map.
\end{lemma}
\begin{proof}  Denote $g=\overline{k}^{-1}$. By the classical Stone duality, $g$ is a continuous proper map. It is enough to show that all the maps $k_F$, $F\in (S/\D)^{\star}$ are injective. Let $F\in (S/\D)^{\star}$. Let, further, $U_1,U_2\in T^{\star}_{g(F)}$ and $V\in S^{\star}_F$ be such that $k_F(U_1)=k_F(U_2)=V$. By the definition of $k_F$, it follows that $k^{-1}(V)=U'$, where $U'\in\PU_{g(F)}$, and $U_1,U_2\subseteq U'$. Let $B_1$ and $B_2$ be  sections of $T^{\star}$ over $U$ such that $U_1\in B_1$ and $U_2\in B_2$, that is $B_1(g(F))=U_1$ and $B_2(g(F))=U_2$. Then the intersection $B_1\cap B_2$ is a section of $T^{\star}$ over some $V$ such that $g(F)\not\in V$ and $k_F((B_1\cap B_2)(g(F)))$ is a section of $S^{\star}$ over some $V'$ such that $F\in V'$. This shows that $V'\neq g^{-1}(V)$, which contradicts the fact that $k$ is an \'{e}tale space cohomomorphism. Therefore, $k_F$ is injective.
\end{proof}

\begin{lemma}\label{31jan2} Let $(E,e,X)$ and $(G,g,Y)$  be \'{e}tale spaces and $k:E\rightsquigarrow G$ an injective \'{e}tale space cohomomorphism over a continuous proper  map. Then $k: E^{\star}\to G^{\star}$  preserves finite intersection, and therefore, can be looked at as a proper skew Boolean $\cap$-algebra homomorphism.
\end{lemma}
\begin{proof} The statement follows from injectivity of $k$, Lemma \ref{lem:parord} and the definition of the operation $\cap$.
\end{proof}

The observation at the beginning of this section, Lemmas \ref{31jan1} and \ref{31jan2} and Theorem \ref{th:main}  show that the restrictions ${\bf{ES}}\vert_{\LSBIA}:\LSBIA\to\ESLCBSE$ and ${\SB}\vert_{\ESLCBSE}:\ESLCBSE\to \LSBIA$ establish the equivalence of categories $\LSBIA$ and $\ESLCBSE$, where the natural isomorphism $\beta: 1_{\LSBIA}\to {\SB}\vert_{\LSBIA}\cdot{\bf{ES}}\vert_{\ESLCBSE}$ and $\gamma: 1_{\ESLCBSE}\to {\bf{ES}}\vert_{\ESLCBSE}\cdot{\SB}\vert_{\LSBIA}$ are given as in \eqref{eq:beta} and \eqref{eq:gamma}, which proves Theorem \ref{th:main_int}.

Corollary \ref{cor:main_int} follows now from Theorem \ref{th:main_int} and the classical Stone duality between Boolean algebras and Boolean spaces.

Let $S$ be a left-handed skew Boolean $\cap$-algebra. For every stalk $S^{\star}_x$, $x\in (S/\D)^{\star}$, let $\tilde{S^{\star}_x}$ be the primitive left-handed skew Boolean $\cap$-algebra, whose non-zero $\D$-class equals ${S^{\star}_x}$. Then $S$ can be represented as a subdirect product of $\prod_{x\in (S/\D)^{\star}}\tilde{S^{\star}_x}$ in a similar manner as it is described in the last paragraph of the previous section. Moreover, in the case, when the maximal lattice image of $S$ is a Boolean algebra, this representation is a {\em Boolean product} representation \cite[p.174]{BS}. One can show that this representation is isomorphic to the Boolean product representation given in \cite[4.10]{L3}.

Let finally $S$ be a skew Boolean $\cap$-algebra with finite stalks, whose maximal lattice image is a Boolean algebra. Let $S^{\star}_x$ be a stalk of maximal cardinality, and $\tilde{S^{\star}_x}$ be the corresponding primitive left-handed skew Boolean $\cap$-algebra. From \cite[4.14]{L3} (see also \cite[4.3]{BL}, \cite[5(2)]{Corn}) it follows that finite primitive left-handed skew Boolean $\cap$-algebras are {\em quasi-primal} (for the definition of quasi-primal algebras see \cite{KW},\cite[IV.10]{BS}). This observation implies that the Boolean product representation of $S$, derived above from the representation of Corollary \ref{cor:main_int}, follows from a general universal algebra result of \cite[3.5]{KW}. It would be interesting to try to find a universal algebra background for the duality from Corollary \ref{th:main} or Theorem \ref{th:main}.

\end{document}